\documentclass[12pt]{amsart}

\usepackage[centering, margin={1.0in, 1.0in}, includeheadfoot]{geometry}

\usepackage{mathrsfs}
\usepackage{amsfonts}
\usepackage{tikz}
\usepackage{amssymb,bm}
\usepackage{amsmath}
\usepackage{amsthm}
\usepackage{mathtools}
\usepackage{tabularx}
\usepackage{hyperref}
\usepackage{dirtytalk}

\newtheorem{thm}{Theorem}[section]
\newtheorem{defi}[thm]{Definition}
\newtheorem{lem}[thm]{Lemma}

\newtheorem{prop}[thm]{Proposition}
\newtheorem{rmk}[thm]{Remark}

\newtheorem{eg}[thm]{Example}

\newcommand{\F}{\mathbb{F}_{q}}

\newcommand{\ZN}{\mathbb{Z}_{N}}

\newcommand{\Fn}{\mathbb{F}_{q^n}}
\newcommand{\Tn}{\operatorname{Tr}_{\mathbb{F}_{q^n}/\F}}

\newcommand{\supp}{\operatorname{supp}}

\newcommand{\lcm}{\operatorname{lcm}}

\begin{document}
 \title{A new proof of the Hansen-Mullen irreducibility conjecture}
 
 \author{Aleksandr Tuxanidy and Qiang Wang}

\address{School of Mathematics and Statistics, Carleton
University,
 1125 Colonel By Drive, Ottawa, Ontario, K1S 5B6,
Canada.} 

\email{AleksandrTuxanidyTor@cmail.carleton.ca, wang@math.carleton.ca}

\keywords{irreducible polynomials, primitive polynomials, Hansen-Mullen conjecture, symmetric functions, $q$-symmetric, discrete Fourier transform, finite fields.\\}

\thanks{The research of Qiang Wang is partially supported by NSERC of Canada.}
\date{\today}

 \font\Bbb msbm10 at 12pt
 
\begin{abstract}
We give a new proof of the Hansen-Mullen irreducibility conjecture. The proof relies on an application of a (seemingly new) sufficient condition for the existence of 
elements of degree $n$ in the support of functions on finite fields. 
This connection to irreducible polynomials is made via the least period of the discrete Fourier transform (DFT) of functions with values in finite fields. 
We exploit this relation and prove, in an elementary fashion, that a relevant function related to the DFT of characteristic elementary symmetric functions (which produce the coefficients of characteristic polynomials) 
has a sufficiently large least period (except for some genuine exceptions). This bears a sharp contrast to previous techniques in literature employed to tackle existence
of irreducible polynomials with prescribed coefficients.
\end{abstract}

 \maketitle

 \section{Introduction}
 
 Let $q$ be a power of a prime $p$, let $\F$ be the finite field with $q$ elements, and let $n \geq 2$.
 In 1992, Hansen-Mullen \cite{hansen-mullen} conjectured (in Conjecture B there; see Theorem \ref{thm: hansen-mullen} below) that, except for a few genuine exceptions, 
 there exist irreducible (and more strongly primitive; see Conjecture A) polynomials of degree $n$ over $\F$ with 
 any {\em one} of its coefficients prescribed to any value. Conjecture B (appearing as Theorem \ref{thm: hansen-mullen} below) was proven by Wan \cite{wan} in 1997 for $q> 19$ or $n \geq 36$,  
 with the remaining cases being computationally verified soon after in \cite{ham-mullen}.
In 2006, Cohen \cite{cohen 2006}, particularly building on some of the work of Fan-Han \cite{fan-han} on $p$-adic series, proved 
there exists a primitive polynomial of degree $n \geq 9$ over $\F$ with any one of its coefficients prescribed. 
The remaining cases of Conjecture A were settled by Cohen-Pre\v{s}ern in \cite{cohen-presern 2006, cohen-presern 2008}. 
Cohen \cite{cohen 2006} and Cohen-Pre\v{s}ern \cite{cohen-presern 2006, cohen-presern 2008} 
also gave theoretical explanations for the small cases of $q,n,$ missed out in Wan's original proof \cite{wan}. 
First for a polynomial $h(x) \in \F[x]$ and an integer $w$, we denote by $[x^w]h(x)$ the coefficient of $x^w$ in $h(x)$.

\begin{thm}\label{thm: hansen-mullen}
 Let $q$ be a power of a prime, let $c \in \F$, and let $n \geq 2$ and $w$ be integers with $1 \leq w \leq n$. If $w = n$, assume that $c \neq 0$. 
 If $(n,w,c) = (2,1,0)$, further assume $q$ is odd.
 Then there exists a monic irreducible polynomial 
 $P(x)$ of degree $n$ over $\F$ with $[x^{n-w}] P(x) = c$.
\end{thm}

The Hansen-Mullen conjectures have since been generalized 
to encompass results on the existence of irreducible and particularly primitive polynomials with {\em several} prescribed coefficients 
(see for instance \cite{garefalakis, George paper, George thesis, pollack, Ha} for general irreducibles and \cite{Cohen 2004, han, ren, Shparlinski} for primitives). 
In particular Ha \cite{Ha}, building on some of the work of Pollack \cite{pollack} and Bourgain \cite{Bourgain}, has recently proved that, for large enough $q,n$,  
there exists irreducibles of degree $n$ over $\F$ 
with roughly any $n/4$ coefficients prescribed to any value. This seems to be the current record on the number of {\em arbitrary} coefficients one may prescribe 
to any values in an irreducible polynomial of degree $n$.

The above are existential results obtained through asymptotic estimates. 
However there is also intensive research on the {\em exact} number of irreducible polynomials with some prescribed coefficients.
See for instance \cite{Carlitz, Fitzgerald-Yucas, KMRV, KPW} and references therein for some work in this area. See also \cite{TW} for primitives and $N$-free elements in special cases.

There are some differences of approach in tackling existence questions of either general irreducible or primitive polynomials with prescribed coefficients. 
For instance, when working on irreducibles, and following in the footsteps of Wan \cite{wan}, 
it has been common practice to exploit the $\F[x]$-analogue of Dirichlet's theorem for primes in arithmetic progressions; 
all this is done via Dirichlet characters on $\F[x]$, $L$-series, zeta functions, etc.
See for example \cite{George thesis}. Recently Pollack \cite{pollack} and Ha \cite{Ha}, building on some ideas of Bourgain \cite{Bourgain}, 
applied the circle method to prove the existence of irreducible polynomials with several prescribed coefficients. 
On the other hand, in the case of primitives, the problem is usually approached via $p$-adic rings or fields 
(to account for the inconvenience that Newton's identities ``break down'', in some sense, in fields of positive characteristic) 
together with Cohen's sieving lemma, Vinogradov's characteristic function, etc. (see for example \cite{fan-han, cohen 2006}).
However there is one common feature these methods share, namely, when bounding
the ``error'' terms comprised of character sums, the function field analogue of Riemann's hypothesis (Weil's bound) is used (perhaps without exception here). 
Nevertheless as a consequence of its $O(q^{n/2})$ nature it transpires 
a difficulty in extending the $n/2$ threshold for the number of coefficients one can prescribe in irreducible or particularly primitive polynomials of degree $n$.

As the reader can take from all this, there seems to be a preponderance of the analytic method to tackle the existence problem of irreducibles 
and primitives with several prescribed coefficients.
One then naturally wonders whether other view-points 
may be useful for tackling such problems. As Panario points out in \cite{Koc}, Chapter 6, Section 2.3, p.115, 

\say{{\em The long-term goal here is to provide existence and counting results for
irreducibles with any number of prescribed coefficients to any given values. This
goal is completely out of reach at this time. Incremental steps seem doable, but
it would be most interesting if new techniques were introduced to attack these
problems}}

In this work we take a different approach and give a new proof of the Hansen-Mullen irreducibility conjecture (or theorem), stated in Theorem \ref{thm: hansen-mullen}.
We attack the problem by studying the least period of certain functions related to the discrete Fourier transform (DFT) of characteristic elementary symmetric functions 
(which produce the coefficients of characteristic polynomials).
This bears a sharp contrast to previous techniques in literature.
The proof theoretically explains, in a unified way,
every case of the Hansen-Mullen conjecture. These include the small cases missed out in Wan's original proof \cite{wan},
computationally verified in \cite{ham-mullen}. 
However we should point out that, in contrast, our proof has the disadvantage of not yielding estimates for the number of irreducibles with a prescribed coefficient.
It merely asserts their existence. 
We wonder whether some of the techniques introduced here can be extended to tackle the existence question for several prescribed coefficients, 
but we for now leave this to the consideration of the interested reader.

The proof relies in an application of the sufficient condition in Lemma \ref{lem: factor of deg n}, which follows from that in (i) of the following lemma. 
First for a primitive element $\zeta$ of $\F$ and a function $f : \mathbb{Z}_{q-1} \to \F$, the DFT of $f$ based on $\zeta$ is the function $\mathcal{F}_{\zeta}[f] : \mathbb{Z}_{q-1} \to \F$ given by 
$$
\mathcal{F}_{\zeta}[f](m) = \sum_{j \in \mathbb{Z}_{q-1}} f(j) \zeta^{mj}, \hspace{1em} m \in \mathbb{Z}_{q-1}. 
$$
Here $\mathbb{Z}_{q-1} := \mathbb{Z}/(q-1) \mathbb{Z}$. 
The inverse DFT is given by $\mathcal{F}_\zeta^{-1}[f] = -\mathcal{F}_{\zeta^{-1}}[f]$.
For a function $g : \mathbb{Z}_{q-1} \to \F$, we say that $g$ has least period $r$ if $r$ is the smallest positive integer such that $g(m + \bar{r}) = g(m)$ for all $m \in \mathbb{Z}_{q-1}$.
Let $\Phi_n(x) \in \mathbb{Z}[x]$ be the $n$-th cyclotomic polynomial. For a function $F$ on a set $A$, let $\supp(F) := \{a \in A \ : \ F(a) \neq 0\}$ be the support of $F$.

\begin{lem}\label{lem: revision}
 Let $q$ be a power of a prime, let $n \geq 2$, let $\zeta$ be a primitive element of $\Fn$, let $F : \Fn \to \Fn$, 
 let $f : \mathbb{Z}_{q^n-1} \to \Fn$ be defined by $f(k) = F(\zeta^k)$, and let $r$ be the least period of $\mathcal{F}_\zeta[f]$ 
 (which is the same as the least period of $\mathcal{F}_{\zeta}^{-1}[f]$). Then we have the following results.
 \\
 (i) If $r \nmid (q^n-1)/\Phi_n(q)$, 
 then $\supp(F)$ contains an element of degree $n$ over $\F$;\\
 (ii) If $\supp(F)$ contains an element of degree $n$ over $\F$, then $r \nmid (q^d-1)$ for every positive divisor $d$ of $n$ with $d < n$;\\
 (iii) If $\supp(F)$ contains a primitive element of $\Fn$, then $r = q^n-1$.
 \end{lem}

In particular (i) implies the existence of an irreducible factor of degree $n$ for any polynomial 
$h(x) \in \F[x]$ satisfying a constraint on the least period as follows. 
Here $\Fn^\times$ and $L^\times$ denote the set of all invertible elements in $\Fn$ and $L$ respectively. 

  \begin{lem}\label{lem: factor of deg n}
  Let $q$ be a power of a prime, let $n \geq 2$, let $h(x) \in \F[x]$, and let $L$ be any subfield of $\Fn$ containing the image $h(\Fn^\times)$.
  Define the polynomial
  $$
  S(x) = \left(1 - h(x)^{\#L^\times} \right) \bmod\left( x^{q^n-1} - 1\right) \in \F[x].
  $$
  Write $S(x) = \sum_{i=0}^{q^n-2} s_i x^i$ for some coefficients $s_i \in \F$. 
  If the cyclic sequence $(s_i)_{i=0}^{q^n-2}$ has least period $r$ satisfying $r \nmid (q^n-1)/\Phi_n(q)$, then 
  $h(x)$ has an irreducible factor of degree $n$ over $\F$.
  \end{lem}
 
Note Lemma \ref{lem: factor of deg n} immediately yields the following sufficient condition for a polynomial to be irreducible. 

\begin{prop}\label{prop: irreducible polynomial condition}
  With the notations of Lemma \ref{lem: factor of deg n}, if $h(x) \in \F[x]$ is of degree $n \geq 2$, and the cyclic sequence $(s_i)_{i=0}^{q^n-2}$ of the coefficients of $S(x)$
  has least period $r$ satisfying $r \nmid (q^n-1)/\Phi_n(q)$, then $h(x)$ is irreducible. 
 \end{prop}
 
 To give the reader a flavor for the essence of our proof as an application of Lemma \ref{lem: factor of deg n}, we give the following small example. 
 
 \begin{eg}
  Let $q = 2$, let $n = 4$, and let 
  \begin{align*}
   h(x) &= \sum_{0 \leq i_1 < i_2 \leq 3} x^{2^{i_1} + 2^{i_2} }\\
   &= x^{12} + x^{10} + x^9 + x^6 + x^5 + x^3 \in \mathbb{F}_2[x].
  \end{align*}
 Note that $h(\mathbb{F}_{2^4}) \subseteq \mathbb{F}_2$. 
 In fact, for any $\xi \in \mathbb{F}_{2^4}$, 
 $h(\xi)$ is the coefficient of $x^2$ in the characteristic polynomial of degree $4$ over $\mathbb{F}_2$ with root $\xi$. 
 We may take $L = \mathbb{F}_2$ in Lemma \ref{lem: factor of deg n}; hence $\#L^\times = 1$. Thus
 \begin{align*}
  S(x) &:= \left(1 + h(x)^{\#L^\times}\right) \bmod\left(x^{2^4-1} + 1 \right) = h(x) + 1\\
  &= x^{12} + x^{10} + x^9 + x^6 + x^5 + x^3 + 1 \in \mathbb{F}_2[x].
 \end{align*}
 The cyclic sequence of coefficients $\mathbf{s} = s_0, s_1, \ldots, s_{2^4-2}$ of $S(x) = \sum_{i=0}^{2^4-2}s_i x^i$ is given by
 $$
 \mathbf{s} = 1, 0, 0, 1, 0, 1, 1, 0, 0, 1, 1, 0, 1, 0, 0.
 $$
 One can easily check that the least period $r$ of $\mathbf{s}$ is $r = 2^4-1$, the maximum possible. 
 Because $2^4-1 > (2^4-1)/\Phi_4(2)$, 
 Lemma \ref{lem: factor of deg n} implies that $h(x)$ has an irreducible factor $P(x)$ of degree $4$ over $\mathbb{F}_2$. %Since irreducible polynomials are also characteristic, 
Any root $\xi$ of $P(x)$ must satisfy $h(\xi) = 0$. This is the coefficient of $x^2$ in $P(x)$. 
Hence there exists an irreducible polynomial of degree $4$ over $\mathbb{F}_2$ with its coefficient of $x^2$ being zero. 
Indeed, $x^4 + x + 1$ is one such irreducible polynomial. 
\end{eg}

The rest of this work goes as follows. 
In Section~\ref{section: dft} we review some preliminary concepts regarding the DFT on finite fields, convolution, 
least period of functions on cyclic groups, and cyclotomic polynomials.
In Section~\ref{section: connection between DFT and irreducibles} we study the connection, between the least period of the DFT of functions, and irreducible polynomials.
In particular we explicitly describe in Proposition \ref{prop: period of DFT} the least period of the DFT of functions, as well as prove Lemmas \ref{lem: revision} and \ref{lem: factor of deg n}.
In Section~\ref{section: delta functions} we introduce the characteristic delta functions as the 
DFTs of characteristic elementary symmetric functions. 
We then apply Lemma \ref{lem: factor of deg n} to give a sufficient condition, in Lemma \ref{lem: delta function},
for the existence of an irreducible polynomial with any one of its coefficients prescribed. This is given in terms of the least period of a certain function $\Delta_{w,c}$, 
closely related to the delta functions.
We also review some basic results on $q$-symmetric functions and their convolutions; this will be needed in Section~\ref{section: proof of Hansen-Mullen}.
Finally in Section~\ref{section: proof of Hansen-Mullen} we prove that the $\Delta_{w,c}$ functions have sufficiently large period. The proof of 
Theorem \ref{thm: hansen-mullen} then immediately follows from this.

 \section{Preliminaries}\label{section: dft}
 We recall some preliminary concepts regarding the DFT for finite fields, convolution, least period of functions on cyclic groups,
 and cyclotomic polynomials.
 
 Let $q$ be a power of a prime $p$, let $N \in \mathbb{N}$ such that $N \mid q-1$, and let $\zeta_N$ be a primitive $N$-th root of unity in $\F^*$ 
 (the condition on $N$ guarantees the existence of $\zeta_N$). We shall use the common notation $\mathbb{Z}_N := \mathbb{Z}/ N \mathbb{Z}$.
 Now the DFT based on $\zeta_N$, on the $\F$-vector space of functions $f : \mathbb{Z}_N \to \F$, is 
 defined by
 $$
 \mathcal{F}_{\zeta_N}[f](i) = \sum_{j \in \mathbb{Z}_N} f(j) \zeta_N^{ij}, \hspace{1em} i \in \mathbb{Z}_N.
 $$
 Note $\mathcal{F}_{\zeta_N}$ is a bijective linear operator with 
inverse given by $\mathcal{F}^{-1}_{\zeta_N} = N^{-1} \mathcal{F}_{\zeta_N^{-1}}$. 

For $f,g : \mathbb{Z}_N \to \F$, the convolution of $f,g$ is the function $f \otimes g : \mathbb{Z}_N \to \F$ given by
$$
(f\otimes g)(i) = \sum_{\substack{j + k = i \\ j,k \in \mathbb{Z}_N} } f(j)g(k). 
$$
Inductively, $f_1 \otimes f_2 \otimes \cdots \otimes f_k = f_1 \otimes (f_2 \otimes \cdots \otimes f_k)$ and so
$$
(f_1 \otimes \cdots \otimes f_k)(i) = \sum_{\substack{j_1 + \cdots + j_k = i \\ j_1, \ldots, j_k \in \mathbb{Z}_N}} f_1(j_1) \cdots f_k(j_k).
$$
For $m \in \mathbb{N}$, we let $f^{\otimes m}$ denote the $m$-th convolution power of $f$, that is, the convolution of $f$ with itself, $m$ times. 
The DFT and convolution are related by the fact that 
$$
\prod_{i=1}^k\mathcal{F}_{\zeta_N}[f_i] = \mathcal{F}_{\zeta_N}\left[ \bigotimes_{i=1}^k f_i\right]. 
$$
Since $f, \mathcal{F}_{\zeta_N}[f]$, have values in $\F$ by definition, it follows from the relation above that $f^{\otimes q} = f$.
Convolution is associative, commutative and distributive with identity $\delta_0 : \mathbb{Z}_N \to \{0, 1\} \subseteq \mathbb{F}_p$, the Kronecker delta function 
defined by $\delta_0(i) = 1$ if $i = 0$ and $\delta_0(i) = 0$ otherwise.
We set $f^{\otimes 0 } = \delta_0$.

Next we recall the concepts of a period and least period of a function $f:\mathbb{Z}_N \to \F$. 
For $r \in \mathbb{N}$, we say that $f$ is {\em $r$-periodic} if $f(i ) = f(i + \overline{r})$ for all $i \in \mathbb{Z}_N$. 
Clearly $f$ is $r$-periodic if and only if it is $\gcd(r, N)$-periodic.
The smallest such positive integer $r$ is called the {\em least period} of $f$. Note the least period $r$ satisfies $r \mid R$ whenever $f$ is $R$-periodic.
If the least period of $f$ is $N$, we say that $f$ has {\em maximum least period}.

There are various operations on cyclic functions which preserve the least period. 
For instance the {\em $k$-shift} function $f_k(i) := f(i + k)$ of $f$ has the same least period as $f$. 
The {\em reversal} function $f^*(i) := f(-(1 + i))$ of $f$ also has the same least period. 
Let $\sigma$ be a permutation of $\F$. The function $f^{\sigma}(i) := \sigma(f(i))$ keeps the 
least period of $f$ as well. 

Next we recall a few elementary facts about cyclotomic polynomials. 
For $n \in \mathbb{N}$, the $n$-th cyclotomic polynomial $\Phi_n(x) \in \mathbb{Z}[x]$
 is defined by 
$$
 \Phi_n(x) = \prod_{k \in \left(\mathbb{Z}/n \mathbb{Z} \right)^\times}\left( x - \zeta_n^k\right),
$$
 where $\zeta_n \in \mathbb{C}$ is a primitive $n$-th root of unity and $\left(\mathbb{Z}/n \mathbb{Z} \right)^\times$ denotes the unit group modulo $n$.
 Since $x^n-1 = \prod_{d \mid n}\Phi_d(x)$, the M\"{o}bius inversion formula gives $\Phi_n(x) = \prod_{d \mid n}(x^{n/d} - 1)^{\mu(d)}$, where $\mu$ is the M\"{o}bius function.

For any divisor $m$ of $n$, with $0 < m < n$, we note that 
$$
\dfrac{x^n-1}{x^m - 1} = \dfrac{\prod_{d \mid n} \Phi_d(x)}{\prod_{d \mid m} \Phi_d(x)} = \prod_{\substack{d \mid n\\ d \nmid m}}\Phi_d(x).
$$ 
Hence 
\begin{equation}\label{eqn: cyclotomic is common divisor}
\Phi_n(x) \mid \dfrac{x^n-1}{x^m - 1} \in \mathbb{Z}[x].
\end{equation}
In fact, one can show that for $n \geq 2$,
$$
\Phi_n(q) = \gcd \left\{ \dfrac{q^n-1}{q^d-1} \ : \ 1\leq d\mid n, \ d < n\right\}
$$
and so 
$$
\dfrac{q^n-1}{\Phi_n(q)} = \lcm\left\{q^d - 1 \ : \ 1\leq d\mid n, \ d < n \right\}.
$$
\iffalse
In particular we have the following sufficient condition for an element to be of degree $n$ over $\F$. 

\begin{prop}\label{prop: deg n}
Let $q$ be a power of a prime, let $n \geq 2$, let $\zeta$ be a primitive element of $\Fn$, and let $i$ be an integer. If $\Phi_n(q) \nmid i$, then $\deg_{\F}(\zeta^i) = n$. 
\end{prop}

\begin{proof}
On the contrary, suppose $\zeta^i \in \mathbb{F}_{q^d}$, and hence $\zeta^i \in \mathbb{F}_{q^d}^*$, for some proper divisor $d$ of $n$. 
Since $\mathbb{F}_{q^d}^*$ is a cyclic group with generator $\zeta^{(q^n-1)/(q^d-1) }$, it follows that $i$ is divisible by $(q^n-1)/(q^d-1)$. 
Because $\Phi_n(q) \mid (q^n-1)/(q^d-1)$, we get $\Phi_n(q) \mid i$, a contradiction.
\end{proof}
\fi
Note also that
\begin{align} \label{eqn: cyclotomic inequality}
 \Phi_n(q) &= \left| \Phi_n(q)\right| = \prod_{k \in (\mathbb{Z}/n\mathbb{Z})^{\times}}\left|q - \zeta_n^k\right| \nonumber \\
 &> q-1
\end{align}
for $n \geq 2$,
since $|q-\zeta_n^k | > q-1$ for any primitive $n$-th root $\zeta_n^k \in \mathbb{C}$, whenever $n \geq 2$ 
(as can be seen geometrically by looking at the complex plane) \footnote{The elementary facts in (\ref{eqn: cyclotomic is common divisor}) and (\ref{eqn: cyclotomic inequality}) 
have some historical significance. For instance, these make an appearance in Witt's classical
proof of Wedderburn's theorem that every finite division ring is a field (see Chapter 5 in \cite{Aigner} for example).}.

\section{Least period of the DFT and connection to irreducible polynomials}\label{section: connection between DFT and irreducibles}
 
 In this section we study a connection between the least period of the DFT of a function and irreducible polynomials. 
 We start off by giving an explicit formula in Proposition \ref{prop: period of DFT} for the least period of the DFT of a function 
 $f : \mathbb{Z}_N \to \F$ in terms of the values in its support. 
Then we prove Lemmas \ref{lem: revision} and \ref{lem: factor of deg n}.

 \iffalse
 As a consequence of this and of Proposition \ref{prop: deg n},  we give a sufficient condition, in Theorem \ref{lem: revision}, for a function
 $F : \mathbb{F}_{q^n} \to \mathbb{F}_{q^n}$ to have an element of degree $n$ over $\F$ in its support $\supp(F) := \{y \in \mathbb{F}_{q^n} \ : \  F(y) \neq 0\}$. 
 This condition is the following: The function $f : \mathbb{Z}_{q^n-1} \to \Fn$ given by $f(k) = F(\zeta^k)$, where $\zeta \in \Fn$ is primitive, 
 is such that $\mathcal{F}_\zeta[f]$ has least period $r$ satisfying $r > (q^n-1)/\Phi_n(q)$. 
 Additionally, we give in Theorem \ref{lem: revision} a necessary condition for a primitive element of $\Fn$ 
 to be contained in the support of $F$.
 \fi
 
 First we may identify, in the usual way, elements of $\ZN = \mathbb{Z} / N \mathbb{Z}$ with their canonical representatives in $\mathbb{Z}$ and vice versa.
 In particular this endows $\ZN$ with the natural ordering in $\mathbb{Z}$.   
 We may also sometimes abuse notation and write $a \mid \bar{b}$ for $a \in \mathbb{Z}$ and $\bar{b} \in \ZN $ 
 to state that $a$ divides the canonical representative of $\bar{b}$, and write $a \nmid \bar{b}$ to state the opposite. 
 For an integer $k$ and a non-empty set $A = \{a_1, \ldots, a_s\}$, we write $\gcd(k, A) := \gcd(k, a_1, \ldots, a_s)$.

\begin{prop}\label{prop: period of DFT}
 Let $q$ be a power of a prime, let $N \mid q-1$, let $f : \mathbb{Z}_N \to \F$ and let $\zeta_N$ be a primitive $N$-th root of unity in $\F^*$. 
 The least period of $\mathcal{F}_{\zeta_N}[f]$ and $\mathcal{F}^{-1}_{\zeta_N}[f]$ is given by $N/\gcd(N, \supp(f))$.
\end{prop} 

\begin{proof}
Note that $d := N/\gcd(N, \supp(f))$ is the smallest positive divisor of $N$ with the property that $N/d$ divides every element in $\supp(f)$. 
For the sake of brevity write $\widehat{f} = \mathcal{F}_{\zeta_N}[f]$. 
Now for $i \in \mathbb{Z}_N$ note that
\begin{align*}
 \widehat{f}(i + d) &= \sum_{j \in \mathbb{Z}_N} f(j) \zeta_N^{(i+d)j} = \sum_{k=0}^{d-1} f\left(\frac{N}{d}k\right) \zeta_N^{(i+d) \frac{N}{d}k} = \sum_{k=0}^{d-1} f\left(\frac{N}{d}k\right) \zeta_N^{i\frac{N}{d}k}\\
 &= \widehat{f}(i). 
\end{align*}
Thus if $r$ is the least period of $\widehat{f}$, necessarily $r \leq d$. 

Since $f = \mathcal{F}_{\zeta_N}^{-1}[ \ \widehat{f} \ ]$, then
 $$
 f(i) = N^{-1} \sum_{j \in \mathbb{Z}_N} \widehat{f}(j) \zeta_N^{-ij}, \hspace{2em} i \in \mathbb{Z}_N. 
 $$
 Hence for $i \in \mathbb{Z}_N$ we have
 \begin{align*}
 Nf(i) &= \sum_{j \in \ZN} \widehat{f}(j) \zeta_N^{-ij} =  \sum_{j=0}^{r-1}\widehat{f}(j) \zeta_N^{-ij} + \sum_{j=r}^{2r-1} \widehat{f}(j) \zeta_N^{-ij} + \cdots + \sum_{j = \left(\frac{N}{r}-1\right)r}^{N-1} \widehat{f}(j)\zeta_N^{-ij} \\
 &= \sum_{j=0}^{r-1} \widehat{f}(j) \zeta_N^{-ij} + \sum_{j=0}^{r-1} \widehat{f}(j+r)\zeta_N^{-i(j + r)} + \cdots + \sum_{j=0}^{r-1} \widehat{f}\left(j + \left(\frac{N}{r}-1\right)r \right) \zeta_N^{-i \left(j + \left(\frac{N}{r}-1\right)r \right)}\\
 &= \sum_{j=0}^{r-1} \widehat{f}(j)\zeta_N^{-ij} + \zeta_N^{-ir}\sum_{j=0}^{r-1} \widehat{f}(j)\zeta_N^{-ij} + \cdots + \zeta_N^{-i\left(\frac{N}{r}-1\right)r}\sum_{j=0}^{r-1} \widehat{f}(j)\zeta_N^{-ij}\\
 &= \sum_{k=0}^{\frac{N}{r} - 1} \zeta_N^{-irk } \sum_{j=0}^{r-1}\widehat{f}(j) \zeta_N^{-ij}. 
 \end{align*}
 If $\frac{N}{r} \nmid i$, then $\zeta_N^{-ir} \neq 1$ and $\sum_{k=0}^{\frac{N}{r} - 1} \zeta_N^{-irk } = \frac{\zeta_N^{-iN} - 1}{ \zeta_N^{-ir} - 1} = 0$. 
 It follows that $f(i) = 0$ whenever $\frac{N}{r} \nmid i$. Equivalently, if $f(i) \neq 0$, then $\frac{N}{r} \mid i$. 
 Now the minimality of $d$ implies that $d \leq r$. But $r \leq d$ (see above) now yields $r = d$. 
 
 With regards to the least period of $\mathcal{F}_{\zeta_N}^{-1}[f]$, 
 we know that $\mathcal{F}_{\zeta_N}^{-1}[f] = N^{-1} \mathcal{F}_{\zeta_N^{-1}}[f]$. Since $\zeta_N^{-1}$
 is a primitive $N$-th root of unity in $\F^*$ as well, the previous arguments similarly imply that $\mathcal{F}_{\zeta_N^{-1}}[f]$ has the least period $d$. 
 Then so does the function $N^{-1} \mathcal{F}_{\zeta_N^{-1}}[f]$,
 a non-zero scalar multiple of $\mathcal{F}_{\zeta_N^{-1}}[f]$.
 \end{proof}

In particular, if $\zeta$ is primitive in $\F$ and $F(x) = \sum_{i \in I} a_i x^i \in \F[x]$ for some subset $I \subseteq [0, q-2]$ of integers with each 
$a_i \neq 0$, $i \in I$, then the least period of the $(q-1)$-periodic sequence
$(F(\zeta^i))_{i \geq 0}$ is given by $(q-1)/\gcd(q-1, I)$. We now prove Lemma \ref{lem: revision}. 

  \begin{proof}[{\bf Proof of Lemma \ref{lem: revision}}]
  (i) On the contrary, suppose that $\supp(F)$ contains no element of degree $n$ over $\F$. Then for each $m \in \supp(f)$ there exists a proper divisor $d$ of $n$ with $(q^n - 1)/(q^d-1) \mid m$.
   Since $\Phi_n(q) \mid (q^n-1)/(q^d-1)$ for all proper divisors $d$ of $n$, then $\Phi_n(q) \mid m$ for all $m \in \supp(f)$. Thus for all $k \in \mathbb{Z}_{q^n-1}$, 
   $$
   \hat{f}(k) = \sum_{j \in \mathbb{Z}_{q^n-1}} f(j) \zeta^{kj} = \sum_{a = 1}^{(q^n-1)/\Phi_n(q)} f\left(a \Phi_n(q) \right) \zeta^{k a \Phi_n(q) },
   $$
   where $\hat{f} = \mathcal{F}_\zeta[f]$. Note that $\hat{f}(k + (q^n-1)/\Phi_n(q)) = \hat{f}(k)$ for all $k \in \mathbb{Z}_{q^n-1}$. Thus $\hat{f}$ is $\frac{q^n-1}{\Phi_n(q)}$-periodic. 
   Necessarily the least period of $\hat{f}$ divides $\frac{q^n-1}{\Phi_n(q)}$, a contradiction.
   \\
   \\
   (ii) Assume $\supp(F)$ contains an element of degree $n$ over $\F$. 
   Then there exists $m \in \supp(f)$ with $(q^n-1)/(q^d-1) \nmid m$ for all proper divisors $d$ of $n$. Let $r$ be the least period of $\hat{f}$. By Proposition \ref{prop: period of DFT},
   $(q^n-1)/r \mid m$. Since $(q^n-1)/(q^d-1) \nmid m$, then $r \nmid q^d-1$ for all proper divisors $d$ of $n$.
   \\
   \\
   (iii) Assume $\supp(F)$ contains a primitive element of $\Fn$. Then there exists $k$ relatively prime to $q^n-1$ such that $\bar{k} \in \supp(f)$. 
  Thus $(\mathbb{Z} / (q^n-1) \mathbb{Z})^\times \cap \supp(f) \neq \emptyset$.
  It follows from Proposition \ref{prop: period of DFT} that both $\mathcal{F}_{\zeta}[f]$ and $\mathcal{F}_{\zeta}^{-1}[f]$ have maximum least period $q^n-1$.
 \end{proof}

 Note that, as the following three examples show, the sufficient (respectively necessary) conditions in Lemma \ref{lem: revision} are not necessary (respectively sufficient).
These may possibly be improved in accordance with the needs of whoever wishes to apply these tools. 
Let us start off by showing that the sufficient condition in (i) is not necessary. 
 
 \begin{eg}
  Recall that $(q^n-1)/\Phi_n(q) = \lcm\{q^d-1 \ : \ d \mid n, \ d < n\}$. Pick any $n$ with at least two prime factors. 
  Then $(q^n-1)/\Phi_n(q) \nmid q^d-1$ for all $d \mid n$, $d < n$. Thus $\zeta^{\Phi_n(q)}$ is of degree $n$ over $\F$. 
  Define the function $F : \Fn \to \Fn$ by $F(\zeta^{\Phi_n(q)}) = 1$ and $F(\xi) = 0$ for all other elements $\xi \in \Fn$. 
  Thus $\supp(F)$ contains an element of degree $n$ over $\F$. The associate function $f : \mathbb{Z}_{q^n-1} \to \Fn$ is defined by $f(k) = 1$ if $k = \Phi_n(q)$ and $f(k) = 0$ otherwise.
  By Proposition \ref{prop: period of DFT}, the least period $r$ of $\mathcal{F}_{\zeta}[f]$ is the smallest positive divisor of $q^n-1$ 
  such that $(q^n-1)/r \mid \Phi_n(q)$, since $\supp(f) = \{\Phi_n(q)\}$. This is $r = (q^n-1)/\Phi_n(q)$. Thus we obtain an example of a function which contains an element of degree $n$ over $\F$ 
  in its support but for which the corresponding least period is a divisor of $(q^n-1)/\Phi_n(q)$. 
 \end{eg}

 The following example shows that the necessary condition in (ii) is not sufficient. 
 
 \begin{eg}
  Similarly as before, pick any $n$ with at least two prime factors. Then $(q^n-1)/\Phi_n(q) \nmid q^d-1$ for all $d \mid n$, $d < n$. Define $F : \Fn \to \Fn$ by $F(\zeta^k) = 1$ if 
  $k = (q^n-1)/(q^d-1)$ for some $d \mid n$, $d < n$, and $F(\xi) = 0$ for all other elements $\xi \in \Fn$. Thus $\supp(F)$ has no element of degree $n$ over $\F$. 
  This defines the associate function $f : \mathbb{Z}_{q^n-1} \to \Fn$ of $F$ with $\supp(f) = \{(q^n-1)/(q^d-1) \ : \ d \mid n, \ d < n\}$. 
  Consider the smallest positive divisor $r$ of $q^n-1$, with $(q^n-1)/r \mid (q^n-1)/(q^d-1)$ for all proper divisors $d$ of $n$. Note that $r$ is divisible by each $q^d-1$, for $d \mid n$, $d < n$; 
  it follows $r = \lcm\{q^d-1 \ : \ d \mid n, \ d < n\} = (q^n-1)/\Phi_n(q)$ with $r \nmid q^d-1$ for all $d \mid n$, $d < n$. By Proposition \ref{prop: period of DFT}, $r = (q^n-1)/\Phi_n(q)$ 
  is the least period of $\mathcal{F}_\zeta[f]$. Thus we have constructed a function $F$ with $\supp(F)$ having no element of degree $ n$ over $\F$ but for which the corresponding least period $r$ satisfies
  $r \nmid (q^d-1)$ for all $d \mid n$, $d < n$. 
 \end{eg}

This last example shows that the necessary condition in (iii) is not sufficient. 

\begin{eg}\label{eg: prim}
 Pick $q,n$ such that $q^n-1$ has at least two non-trivial relatively prime divisors, say $a,b > 1$ with $a,b \mid (q^n-1)$ and $\gcd(a,b) = 1$. 
  The smallest positive divisor $r$ of $q^n-1$ with $(q^n-1)/r \mid a,b$
 is $r = q^n-1$. Now we note that the function $F : \Fn \to \Fn$ defined by $F(\zeta^a) = F(\zeta^b) = 1$ and $F(\xi) = 0$ for all other elements $\xi$ of $\Fn$, 
 contains no primitive element in its support, but the corresponding least period of $\mathcal{F}_\zeta[f]$ is $q^n-1$, by Proposition \ref{prop: period of DFT}.
 \end{eg}

\begin{rmk}
 %In passing note the fact that $r = q^n-1$ in
 We remark that  Example \ref{eg: prim} together with Lemma \ref{lem: revision} (i) 
 imply that for any such $a,b$, there exists $k \in \{a,b\}$ such that $(q^n-1)/(q^d-1) \nmid k $ for all proper divisors $d$ of $n$; that is, 
 either $\zeta^a$ or $\zeta^b$ (or both) is an element of degree $n$ over $\F$. This may also have applications in determining whether a polynomial $h(x) \in \F[x]$ 
 has an irreducible factor of degree $n$.
Specifically, if there exist divisors $a,b \geq 1$ of $q^n-1$ with $\gcd(a,b) = 1$ and $h(\zeta^a) = h(\zeta^b) = 0$, then $h(x)$ has an irreducible factor of degree $n$. 
\end{rmk}

Finally we prove Lemma~\ref{lem: factor of deg n}. 

 \begin{proof}[{\bf Proof of Lemma \ref{lem: factor of deg n}}]
   As a function on $\Fn^\times$, note that 
 $$
 S(\xi) = \begin{cases}
             1 & \mbox{ if } h(\xi) = 0\\
             0 & \mbox{ otherwise.}
            \end{cases}
$$
Let $\zeta$ be a primitive element of $\Fn$ and define the function $f : \mathbb{Z}_{q^n-1} \to \F$ by $f(m) = s_m$. 
Thus $f$ has least period $r$ satisfying $r \nmid (q^n-1)/\Phi_n(q)$.
Note that $S(\zeta^i) = \sum_{j} s_j \zeta^{ij} = \sum_j f(j) \zeta^{ij} = \mathcal{F}_\zeta[f](i)$ for each $i \in \mathbb{Z}_{q^n-1}$. 
Then by the criteria (i) of Lemma \ref{lem: revision}, 
there exists an element of degree $n$ over $\F$ in the support of $S$. It follows $h(x)$ has a root of degree $n$ over $\F$ and hence has an irreducible factor of degree $n$ over $\F$.
 \end{proof}

\section{Characteristic elementary symmetric and delta functions}\label{section: delta functions}

In this section we apply Lemma \ref{lem: factor of deg n} for the purposes of studying coefficients of irreducible polynomials. 
We first place the characteristic elementary symmetric functions in the context of their DFT, 
which we shall refer to here simply as delta functions. These delta functions are indicators, with values in a finite field, 
for sets of values in $\mathbb{Z}_{q^n-1}$ whose 
canonical integer representatives have certain Hamming weights in their $q$-adic representation 
and $q$-digits all belonging to the set $\{0,1\}$.
Essentially, characteristic elementary symmetric functions are characteristic generating functions of the sets that the delta functions indicate. 
Then we give in Lemma \ref{lem: delta function}
sufficient conditions for an irreducible polynomial to have a prescribed coefficient. 
Because the delta functions are $q$-symmetric (see Definition \ref{def: q-sym}), 
we also review some useful facts needed in Section~\ref{section: proof of Hansen-Mullen}.

For $\xi \in \Fn$, the characteristic polynomial $h_\xi(x) \in \F[x]$ of degree $n$ over $\F$ with root $\xi$ is given by 
$$
h_\xi(x) = \prod_{k=0}^{n-1}\left( x - \xi^{q^k} \right) = \sum_{w=0}^{n}(-1)^w \sigma_w(\xi) x^{n-w},
$$
where for $0 \leq w \leq n$, $\sigma_w(x) \in \F[x]$ is the {\em characteristic elementary symmetric} polynomial
given by $\sigma_0(x) = 1$ and
$$
 \sigma_w(x) = \sum_{0 \leq i_1 <  \cdots < i_w \leq n-1} x^{q^{i_1} + \cdots + q^{i_w}}, 
 $$
 for $ 1 \leq w \leq n$.
 In particular $\sigma_1 = \Tn$ is the (linear) trace function and $\sigma_n = N_{\Fn/\F}$ is the (multiplicative) norm function. 
 Whenever $q = 2$ and $\xi \neq 0$, then $\sigma_0(\xi) = \sigma_n(\xi) = 1$ always. 
 If $\xi \neq 0$, then (in general) $h_{\xi^{-1}}(x) = (-1)^n \sigma_n(\xi^{-1}) x^n h_\xi(1/x) = h_\xi^*(x)$, where $h_\xi^*(x)$ is the (monic) {\em reciprocal} of $h_\xi(x)$. Thus
 $\sigma_w(\xi) = \sigma_n(\xi) \sigma_{n-w}(\xi^{-1})$. Clearly $h_\xi(x)$ is irreducible if and only if so is $h_{\xi}^*(x)$. 
 This occurs if and only if $\deg_{\F}(\xi) = n$.
 
Next we introduce the characteristic delta functions and the sets they indicate. But first let us clarify some ambiguity in our notation: 
For $a,b \in \mathbb{Z}$, we denote by $a \bmod b$ the remainder of division of $a$ by $b$.
That is, $a \bmod b$ is the smallest integer $c$ in $\{0, 1, \ldots, b-1\}$ that is congruent to $a$ modulo $b$, and write $c = a\bmod b$. 
Similarly if $\bar{a} = a + b\mathbb{Z}$ is an element of $\mathbb{Z}_b$, we use the notation $\bar{a} \bmod b := a \bmod b$ to express the canonical representative of $\bar{a}$ in $\mathbb{Z}$.
But we keep the usual notation $k \equiv a \pmod{b}$ to state that $b \mid (k-a)$.

We can represent $a \in \mathbb{Z}_{q^n-1}$ uniquely by the $q$-adic representation $(a_0, \ldots, a_{n-1})_q = \sum_{i=0}^{n-1}a_i q^i$, with each $0\leq a_i \leq q-1$, 
of the canonical representative of $a$ in $\{0, 1, \ldots, q^n-2\} \subset \mathbb{Z}$.
For the sake of convenience we write $a = (a_0, \ldots, a_{n-1})_q$.
For $w \in [0, n] := \{0, 1, \ldots, n\}$, define the sets $\Omega(w) \subseteq \mathbb{Z}_{q^n-1}$ by $\Omega(0) = \{0\}$ and 
 $$
 \Omega(w) = \left\{ k \in \mathbb{Z}_{q^n-1} \ : \ k \bmod (q^n-1) = q^{i_1} + \cdots + q^{i_w}, \ 0 \leq i_1 <  \cdots < i_w \leq n-1              \right\}
$$
for $1 \leq w \leq n$. That is, $\Omega(w)$ consists of all the elements $k \in \mathbb{Z}_{q^n-1}$ whose canonical representatives in $\{0, 1, \ldots, q^n-2\} \subset \mathbb{Z}$ 
have Hamming weight $w$ in their 
$q$-adic representation $(a_0, \ldots, a_{n-1})_q = \sum_{i=0}^{n-1} a_i q^i$, 
with each $a_i \in \{0,1\}$.
Note this last condition that each $a_i \in \{0,1\}$ is automatically redundant when $q = 2$, 
since in general each $a_i \in [0, q-1]$ in the $q$-adic representation $t = (a_{0}, \ldots, a_m)_q$ of
a non-negative integer
$
t = \sum_{i=0}^{m}a_i q^i.
$

When $q = 2$, note $\Omega(n) = \emptyset$ since there is no integer in $\{0, 1, \ldots, 2^n-2\}$ with Hamming weight $n$ in its binary representation. 
Observe also that $|\Omega(w)| = {n \choose w}$ for each $0 \leq w \leq n$, unless $(q,w) = (2,n)$. 
Moreover $\Omega(v) \cap \Omega(w) = \emptyset$ whenever $v \neq w$, by the uniqueness of base representation of integers. 
For $w \in [0, n]$, define the characteristic (finite field valued) function
$\delta_w : \mathbb{Z}_{q^n-1} \to \mathbb{F}_p$ of the set $\Omega(w)$ by
$$
\delta_w(k) = \begin{cases}
              1 & \mbox{ if } k \in \Omega(w);\\
              0 & \mbox{ otherwise.}
              \end{cases}
              $$
Observe that our $\delta_0$ is the Kronecker delta function on $\mathbb{Z}_{q^n-1}$ 
with values in $\{0,1\} \subseteq \mathbb{F}_p$.

\begin{lem}\label{lem: sigma delta}
 Let $\zeta$ be a primitive element of $\Fn$ and let $w \in [0, n]$. If $q = 2$, further assume that $w \neq n$. Then   
 $$
 \sigma_w(\zeta^k) = \mathcal{F}_\zeta[\delta_w](k), \hspace{2em} k \in \mathbb{Z}_{q^n-1}.
 $$
\end{lem}

\begin{proof}
Note $\sigma_0(\zeta^k) = 1$ for each $k$ and so $\sigma_0(\zeta^k) = \mathcal{F}_\zeta[\delta_0](k)$. 
Now let $ 1 \leq w \leq n$.
 By definition and the assumption that $(q,w) \neq (2,n)$, we have
\begin{align*}
 \sigma_w(\zeta^k) &= \sum_{0 \leq i_1 <  \cdots < i_w \leq n-1} \zeta^{k\left(q^{i_1} + \cdots  + q^{i_w}\right)} = \sum_{j \in \mathbb{Z}_{q^n-1}} \delta_w(j) \zeta^{kj} \\
 &= \mathcal{F}_\zeta[\delta_w](k) .
\end{align*}
 \end{proof}

These functions are related to various mathematical objects in literature: 
Let $m < q$, let $r_1, \ldots, r_m \in [1, n-1]$, and let $c_0, \ldots, c_{n-1} \in [0, m-1]$
such that $\sum_{i=1}^m r_i = \sum_{j=0}^{n-1} c_j$. View each $\delta_{r_1}, \ldots, \delta_{r_m}$ as having values in $\mathbb{Z}$.
Then one can show that 
$$\delta_{r_1} \otimes \cdots \otimes \delta_{r_m}((c_0, \ldots, c_{n-1})_q)$$
is the number of $m \times n$ matrices, 
with entries in $\{0, 1\} \subset \mathbb{Z}$, 
such that the sum of the entries in row $i$, $1 \leq i \leq m$, is $r_i$, and the sum of the entries in column $j$, $0 \leq j \leq n-1$, is $c_j$.
Matrices with 0--1 entries and prescribed row and column sums are classical
objects appearing in numerous branches of pure and applied mathematics, such as combinatorics, algebra and statistics.
See for instance the survey in \cite{Barvinok} and Chapter 16 in \cite{Lint}.

 An application of Lemma \ref{lem: factor of deg n} yields the following sufficient condition for the existence of irreducible polynomials with a prescribed coefficient.
 
 \begin{lem}\label{lem: delta function}
  Fix a prime power $q$ and integers $n \geq 2$ and $1 \leq w \leq n$. Fix $c \in \F$. If $q = 2$, further assume that $w \neq n$. 
  If the function $\Delta_{w,c} : \mathbb{Z}_{q^n-1} \to \F$ given by 
  $$
  \Delta_{w,c} = \delta_0 - ((-1)^w \delta_w - c \delta_0)^{\otimes(q-1)}
  $$
  has least period $r$ satisfying $r \nmid (q^n-1)/\Phi_n(q)$, then there exists an irreducible polynomial $P(x)$ of degree $n$ over $\F$ with $[x^{n-w}]P(x) = c$.
 \end{lem}
 
 \begin{proof}
  Take $h(x) = (-1)^w\sigma_w(x) - c \in \F[x]$ in Lemma \ref{lem: factor of deg n}. 
  Since $\sigma_w(\Fn) \subseteq \F$, we can pick $L = \F$. Thus $S(x) \in \F[x] $ is given by 
  $$
  S(x) = \left[1 - ((-1)^w \sigma_w(x) - c)^{q-1}\right] \bmod\left(x^{q^n-1} - 1\right). 
  $$ 
  Let $\zeta$ be a primitive element of $\Fn$. 
  By Lemma \ref{lem: sigma delta}, the linearity of the DFT, and the fact that $c = \mathcal{F}_\zeta[c \delta_0]$, we have
 $$
   S(\zeta^i) = 1 - \left((-1)^w \sigma_w(\zeta^i) - c\right)^{q-1} = \mathcal{F}_\zeta[\delta_0](i) - \left(\mathcal{F}_\zeta\left[(-1)^w\delta_w - c \delta_0\right](i) \right)^{q-1}.
 $$
Since the product of DFTs is the DFT of the convolution, then, as a function on $\Fn$,
\begin{align*}
S &= \mathcal{F}_\zeta[\delta_0] - \mathcal{F}_\zeta\left[\left((-1)^w\delta_w - c \delta_0\right)^{\otimes(q-1)}\right]
= \mathcal{F}_\zeta\left[ \delta_0 - \left((-1)^w\delta_w - c \delta_0\right)^{\otimes(q-1)} \right]\\
&= \mathcal{F}_\zeta[\Delta_{w,c}].
\end{align*}
Thus
$$
S(\zeta^m) = \sum_{i=0}^{q^n-2}\Delta_{w,c}(i) \zeta^{mi}
$$
for each $m \in \mathbb{Z}_{q^n-1}$.
As $S(x)$ is already reduced modulo $x^{q^n-1} - 1$, it follows (from the uniqueness of the DFT of a function) that $S(x) = \sum_{i=0}^{q^n-2}\Delta_{w,c}(i) x^i$. 
Since the least period of $\Delta_{w,c}$
is not a divisor of $(q^n-1)/\Phi_n(q)$ by assumption, 
Lemma \ref{lem: factor of deg n} implies $h(x)$ has an irreducible factor $P(x)$ of degree $n$ over $\F$. 
Any of the roots $\xi$ of $P(x)$ must satisfy $h(\xi) = 0$, that is, $(-1)^w \sigma_w(\xi) = c$. 
This is the coefficient of $x^{n-w}$ in $P(x)$. Hence $[x^{n-w}]P(x) = c$ with $P(x)$
irreducible of degree $n$ over $\F$. 
\end{proof}

Note the delta functions also satisfy the property that 
\begin{equation}\label{eqn: sym property}
\delta_{w}((a_0, \ldots, a_{n-1})_q) = \delta_{w}((a_{\rho(0)}, \ldots, a_{\rho(n-1)})_q)
\end{equation}
for every permutation $\rho$ of the indices in $[0, n-1]$. In particular such functions have a natural well-studied dyadic analogue in the 
{\em symmetric boolean} functions. These are boolean functions $f : \mathbb{F}_{2}^n \to \mathbb{F}_2$ with the property that 
$f(x_0, \ldots, x_{n-1}) = f(x_{\rho(0)}, \ldots, x_{\rho(n-1)})$ for every permutation $\rho \in \mathcal{S}_{[0, n-1]}$;
hence the value of $f(x_0, \ldots, x_{n-1})$ depends only on the Hamming weight of $(x_0, \ldots, x_{n-1})$. 
See for example \cite{Canteaut, Castro} for some works on symmetric boolean functions. 
Nevertheless in our case the domain of these $\delta_w$ functions 
is $\mathbb{Z}_{q^n-1}$ rather than $\mathbb{F}_{2}^n$. Although one may still represent the elements of $\mathbb{Z}_{q^n-1}$
as $n$-tuples, say by using the natural $q$-adic representation, the arithmetic here is not as nice as in $\mathbb{F}_{2}^n$.
One has to consider the possibility that a ``carry'' may occur when adding or subtracting (this can make things quite chaotic)
and also worry about reduction modulo $q^n-1$ (although this is much easier to deal with). 
These issues will come up again in the following section. 
The symmetry property in (\ref{eqn: sym property}) of 
$\delta_w$ and of its convolutions will be exploited
in the proof of Lemma \ref{lem: period of delta and hansen-mullen} for the case when $(w,c) = (n/2, 0)$. 

Before we move on to the following section, we need the fact in Lemma \ref{lem: convolution of q-sym is q-sym}.
First for a permutation $\rho \in \mathcal{S}_{[0, n-1]}$ of the indices in the set $[0, n-1]$, define the map $\varphi_\rho : \mathbb{Z}_{q^n-1} \to \mathbb{Z}_{q^n-1}$ by 
\begin{equation}\label{eqn: base q bijection}
\varphi_\rho((a_0, \ldots, a_{n-1})_q) = (a_{\rho(0)}, \ldots, a_{\rho(n-1)})_q.
\end{equation}
Note $\varphi_{\rho}$ is a permutation of $\mathbb{Z}_{q^n-1}$ with inverse $\varphi_\rho^{-1} = \varphi_{\rho^{-1}}$, for each $\rho \in \mathcal{S}_{[0, n-1]}$. For $k \in \mathbb{Z}_{q^n-1}$, let $\epsilon_i(k)$, $0 \leq i \leq n-1$,
denote the digit of $q^i$ in the $q$-adic form of 
its canonical representative. Thus $0 \leq \epsilon_i(k) \leq q-1$. For $a,b \in \mathbb{Z}_{q^n-1}$ with $a +b \neq 0$, it is clear that if $\epsilon_i(a) + \epsilon_i(b) \leq q-1$, then $\epsilon_i(a + b) = \epsilon_i(a) + \epsilon_i(b)$.
One can also check, for any $a,b \in \mathbb{Z}_{q^n-1}$ such that $\epsilon_i(a) + \epsilon_i(b) \leq q-1$ holds for every $0 \leq i \leq n-1$, that 
$\varphi_\rho(a + b) = \varphi_\rho(a) + \varphi_\rho(b)$ for every $\rho \in \mathcal{S}_{[0, n-1]}$, regardless of whether $a + b = 0$ or not. 
By induction, $\varphi_\rho(a_1 + \cdots + a_s) = \varphi_\rho(a_1) + \cdots + \varphi_\rho(a_s)$, whenever $a_1, \ldots, a_s \in \mathbb{Z}_{q^n-1}$ satisfy 
$\epsilon_i(a_1) + \cdots + \epsilon_i(a_s) \leq q-1$
for every $0\leq i \leq n-1$.

\begin{defi}[{\bf q-symmetric}]\label{def: q-sym}
 For a function $f$ on $\mathbb{Z}_{q^n-1}$, we say that $f$ is {\em $q$-symmetric} if for all $a = (a_0, \ldots, a_{n-1})_q \in \mathbb{Z}_{q^n-1}$ and all permutations 
$\rho \in \mathcal{S}_{[0, n-1]}$, we have $f(\varphi_\rho(a)) = f(a)$; that is,
$$f((a_{\rho(0)}, \ldots, a_{\rho(n-1)})_q) = f((a_0, \ldots, a_{n-1})_q). $$
\end{defi}

Note the $\delta_w$ functions are $q$-symmetric. Because $\epsilon_i(m) \leq 1$ for each $m \in \supp(\delta_w) = \Omega(w)$ and each $0 \leq i \leq n-1$, 
it follows from the following lemma that the convolution of at most $q-1$ delta functions is also $q$-symmetric.

\begin{lem}\label{lem: convolution of q-sym is q-sym}
 Let $R$ be a ring and let $f_1, \ldots, f_s : \mathbb{Z}_{q^n-1} \to R$ be $q$-symmetric functions such that for each $a_k \in \supp(f_k)$, $1 \leq k \leq s$, we have
 $\epsilon_i(a_1) + \cdots + \epsilon_i(a_s) \leq q-1$ for every $0 \leq i \leq n-1$. Then $f_1 \otimes \cdots \otimes f_s$ is $q$-symmetric. 
\end{lem}

\begin{proof}
Recall the assumption on the supports imply that $\varphi_\tau(a_1 + \cdots + a_s) = \varphi_\tau(a_1) + \cdots + \varphi_\tau(a_s)$ 
for any $a_k \in \supp(f_k)$, $1 \leq k \leq s$, and any $\tau \in \mathcal{S}_{[0,n-1]}$. 
Since each $f_k$ is $q$-symmetric, $1 \leq k \leq s$, then $f_k(a) = f_k(\varphi_\tau(a))$ for every $a \in \mathbb{Z}_{q^n-1}$. 
In particular $a \in \supp(f_k)$ if and only if $\varphi_\tau(a) \in \supp(f_k)$; hence $\varphi_\tau(\supp(f_k)) = \supp(f_k)$.
Now let $m \in \mathbb{Z}_{q^n-1}$ and let $\rho \in \mathcal{S}_{[0, n-1]}$. Then it follows from the aforementioned observations that
\begin{align*}
 (f_1 \otimes \cdots \otimes f_s)(\varphi_\rho(m)) 
 &= \sum_{\substack{j_1 + \cdots + j_s = \varphi_\rho(m) \\ j_1, \ldots, j_s \in \mathbb{Z}_{q^n-1}}} f_1(j_1) \cdots f_s(j_s)\\
  &= \sum_{\substack{j_1 + \cdots + j_s = \varphi_\rho(m) \\ j_1 \in \supp(f_1), \ldots, j_s \in \supp(f_s)}} f_1(j_1) \cdots f_s(j_s)\\
  &= \sum_{\substack{\varphi_{\rho^{-1}}(j_1 + \cdots + j_s) = m \\ j_1 \in \supp(f_1), \ldots, j_s \in \supp(f_s)}} f_1(j_1) \cdots f_s(j_s)\\
  &= \sum_{\substack{\varphi_{\rho^{-1}}(j_1) + \cdots + \varphi_{\rho^{-1}}(j_s) = m \\ j_1 \in \supp(f_1), \ldots, j_s \in \supp(f_s)}} f_1(j_1) \cdots f_s(j_s)\\
  &= \sum_{\substack{j_1 + \cdots + j_s = m \\ \varphi_\rho(j_1) \in \supp(f_1), \ldots, \varphi_\rho(j_s) \in \supp(f_s)}} f_1(\varphi_\rho(j_1)) \cdots f_s(\varphi_\rho(j_s))\\
  &= \sum_{\substack{j_1 + \cdots + j_s = m \\ \varphi_\rho(j_1) \in \supp(f_1), \ldots, \varphi_\rho(j_s) \in \supp(f_s)}} f_1(j_1) \cdots f_s(j_s)\\
  &= \sum_{\substack{j_1 + \cdots + j_s = m \\ j_1 \in \varphi_{\rho^{-1}}(\supp(f_1)), \ldots, j_s \in \varphi_{\rho^{-1}}(\supp(f_s))}} f_1(j_1) \cdots f_s(j_s)\\
  &= \sum_{\substack{j_1 + \cdots + j_s = m \\ j_1 \in \supp(f_1), \ldots, j_s \in \supp(f_s)}} f_1(j_1) \cdots f_s(j_s)\\
  &= \sum_{\substack{j_1 + \cdots + j_s = m \\ j_1, \ldots, j_s \in \mathbb{Z}_{q^n-1}}} f_1(j_1) \cdots f_s(j_s)\\
  &= (f_1 \otimes \cdots \otimes f_s)(m),
 \end{align*}
as required.
\end{proof}

\section{Least period of $\Delta_{w,c}$ and proof of Theorem \ref{thm: hansen-mullen}}\label{section: proof of Hansen-Mullen}

In this section we prove in Lemma \ref{lem: period of delta and hansen-mullen} that the $\Delta_{w,c}$ function of Lemma \ref{lem: delta function} 
has least period larger than $(q^n-1)/\Phi_n(q)$, at least in the cases 
that suffice for a proof of Theorem \ref{thm: hansen-mullen}. 
Note the proof of Lemma \ref{lem: period of delta and hansen-mullen} is of a rather elementary and constructive type nature.
We then conclude the work with an immediate proof of Theorem~\ref{thm: hansen-mullen}. 
First for an integer $k = \sum_{i=0}^\infty \epsilon_i(k) q^i$, we let $s_q(k) = \sum_{i=0}^\infty \epsilon_i(k)$ denote the sum of the $q$-digits of $k$. 

\begin{lem}\label{lem: period of delta and hansen-mullen}
 Let $q$ be a power of a prime, let $n \geq 2$, let $w$ be an integer with $1 \leq w \leq n/2$, and let $c \in \F$. 
 If $c = 0$ and $n = 2$, further assume that $q$ is odd. Then the least period $r$ of $\Delta_{w,c}$ satisfies $r > (q^n-1)/\Phi_n(q)$.
 More precisely, we have the following three results:
 \\
 (i) If $c \neq 0$, or $c = 0$ and $w \neq n/2$, or $n > 2$ and $q$ is even with $(w,c) = (n/2, 0)$, then $r = q^n-1$;
 \\
 (ii) If $c = 0$, $q$ is odd, $n > 2$ and $w = n/2$, then $r \geq (q^n- 1)/2$;
 \\
 (iii) If $c = 0$, $w = 1$, $q$ is odd and $n = 2$, then $r > q-1$.
\end{lem}

\begin{proof}
We shall suppose that $0 < r < q^n-1$ is a period of $\Delta_{w,c}$ and either aim to obtain a contradiction or show that $r \geq (q^n-1)/2$ or $r > q-1$ as required, in accordance with the 
cases in (i), (ii), (iii).
Now since $\Delta_{w,c}$ is $r$-periodic, $\Delta_{w,c}(m) = \Delta_{w,c}(m \pm r)$ for all $m \in \mathbb{Z}_{q^n-1}$.
Because $q^n-1 = \sum_{i=0}^{n-1}(q-1)q^i$, 
we may write $r = \sum_{i=0}^{n-1} r_i q^i$
for some $q$-digits $r_i$ with each $0\leq r_i \leq q-1$, not all $r_i = q-1$, for $0 \leq i \leq n-1$. 
\\
\\
{\bf Case 1 ($c \neq 0$):} Assume $c \neq 0$. We shall prove that $\Delta_{w,c}$ has maximum least period in this case. On the contrary, suppose that $r$ is a period of $\Delta_{w,c}$ 
with $0 < r < q^n-1$. 
By the binomial theorem for convolution,
\begin{align*}
\left((-1)^w\delta_w -c\delta_0\right)^{\otimes(q-1)} &= \sum_{s=0}^{q-1} {q-1 \choose s} (-c)^{q-1-s} (-1)^{ws} \delta_w^{\otimes s}\\
&= \sum_{s=0}^{q-1} {q-1 \choose s} ((-1)^{w+1} c)^{-s} \delta_w^{\otimes s}.
\end{align*}
Hence 
\begin{align*}
\Delta_{w,c} &:= \delta_0 - \left((-1)^w\delta_w -c\delta_0\right)^{\otimes(q-1)}\\
&= -\sum_{s=1}^{q-1} {q-1 \choose s} ((-1)^{w+1} c)^{-s} \delta_w^{\otimes s}.
\end{align*}
By Lucas' theorem, none of the binomial coefficients above are $0$ modulo $p$, where $p$ is the characteristic of $\F$.
Now note for any $m \in \mathbb{Z}_{q^n-1}$ that $\delta_w^{\otimes s}(m)$ is the number, modulo $p$, of ways to write $m$ as a sum of $s$ ordered values in $\Omega(w)$. 
We avoid dealing with complicated expressions for this number; instead let us note a few simpler facts:

(a) For $1 \leq s \leq q-1$, there occurs no carry in the $q$-adic addition of any $s$ non-negative integers with $q$-digits at most $1$. In particular, viewing $\Omega(w)$ 
as lying in $\{0, 1, \ldots, q^n-2\} \subset \mathbb{Z}$ in the natural way, we conclude there occurs no carry in the addition of any 
$s$ values in $\Omega(w)$, when $1 \leq s \leq q-1$. Since $w < n$ as well, any such addition of $1 \leq s \leq q-1$ elements in $\Omega(w)$ is strictly smaller than $q^n-1$. 

(b) $m \in \supp(\Delta_{w,c})$ if and only if there exists a unique $1 \leq s \leq q-1$ such that $m \in \supp(\delta_w^{\otimes s})$. 
In particular if $m \in \supp(\Delta_{w,c})$, then $s_q(m \bmod(q^n-1)) \leq (q-1)w $. 
Indeed, if $m \in \supp(\delta_w^{\otimes s})$ for $s$ with $1 \leq s \leq q-1$, then $s_q(m \bmod (q^n-1)) = sw$. 
Hence for every $k \neq s$ with $1 \leq k \leq q-1$, we get $\delta_w^{\otimes k}(m) = 0$ since $kw \neq sw$. 

(c) For any $1 \leq s \leq q-1$ and $t \in \Omega(w)$, 
we have $\delta_w^{\otimes s}(st) = 1$. Indeed, it is not hard to see there is exactly one way to write $st$ as a sum of $s$ values in $\Omega(w)$, 
namely as $st = t + \cdots + t$, $s$ times. It follows for every $s$ with $1 \leq s \leq q-1$ and every $t \in \Omega(w)$, that $st \in \supp(\Delta_{w,c})$.

\iffalse
First observe for $1 \leq s \leq q-1$, there occurs no carry in the $q$-ary addition of any $s$ non-negative integers with $q$-digits at most $1$. In particular, viewing $\Omega(w)$ 
as lying in $\{0, 1, \ldots, q^n-2\} \subset \mathbb{Z}$ in the natural way, we conclude there occurs no carry in the addition of any $s$ values in $\Omega(w)$, when $1 \leq s \leq q-1$. 
Since $w < n$ as well, any such addition of $1 \leq s \leq q-1$ elements in $\Omega(w)$ is strictly smaller than $q^n-1$.
Consequently if $m \in \supp(\delta_w^{\otimes s})$ for $s$ with $1 \leq s \leq q-1$, then $s_q(m \bmod (q^n-1)) = sw$.
In this case, for every $k \neq s$ with $1 \leq k \leq q-1$, we get $\delta_w^{\otimes k}(m) = 0$ since $kw \neq sw$. 
It follows that if $m \in \supp(\Delta_{w,c})$, then there exists a unique $s$, $1 \leq s \leq q-1$, for which $m \in \supp(\delta_w^{\otimes s})$; in this case
$s_q(m \bmod(q^n-1)) = sw$. On the other hand, if $m \in \supp(\delta_w^{\otimes s})$ for some $1 \leq s \leq q-1$, then $m \in \supp(\Delta_{w,c})$. 
In particular if $m \in \supp(\Delta_{w,c})$, then $s_q(m \bmod(q^n-1)) \leq (q-1)w $.
Next note that for any $1 \leq s \leq q-1$ and $t \in \Omega(w)$, 
that $\delta_w^{\otimes s}(st) = 1$. Indeed, it is not hard to see there is exactly one way to write $st$ as a sum of $s$ values in $\Omega(w)$, 
namely as $st = t + \cdots + t$, $s$ times. It follows for every $s$ with $1 \leq s \leq q-1$ and every $t \in \Omega(w)$, that $st \in \supp(\Delta_{w,c})$. 
\fi

Having gathered a few facts about $\Delta_{w,c}$, we proceed with the proof: 
Clearly either $s_q(r) \leq (q-1)n/2$ or $s_q(r) > (q-1)n/2$. Suppose $s_q(r) \leq (q-1)n/2$. Let $M := \{i \in [0, n-1] \ : \ r_i = q-1\}$ and let $\eta := \#M$. Clearly $s_q(r) \geq (q-1)\eta$.
It is impossible that $\eta > n-w$. Indeed, since $w \leq n/2$, we would have $\eta > n/2$ and $s_q(r) > (q-1)n/2$, a contradiction.
Hence $\eta \leq n-w$ and so $w \leq n - \eta$.
Then there exists a subset $W \subseteq [0, n-1] \setminus M$ with $\#W = w$.
Let $\mathcal{C}$ be the collection of all such subsets $W$. Thus $\mathcal{C} \neq \emptyset$ and $\max_{i \in W} \{r_i\} \leq q-2$ for all $W \in \mathcal{C}$.

We claim there exists $W \in \mathcal{C}$ such that
\begin{equation}\label{eqn: claim 0}
s_q(r) > w \max_{i \in W}\{r_i\}.
\end{equation}
Indeed, suppose on the contrary that
\begin{equation}\label{eqn: claim 1}
s_q(r) \leq w \max_{i \in W}\{r_i\} \hspace{1em} \text{ for all } W \in \mathcal{C}.
\end{equation}
Thus for every $W \subseteq [0, n-1] \setminus M$, with $\#W = w$, there exists $i \in W$ such that $r_i \geq s_q(r)/w$.
In particular $s_q(r) \leq (q-2)w$.
As $\#([0, n-1]\setminus M) = n - \eta$ and each $\#W = w$, it follows from (\ref{eqn: claim 1})
that $r$ has at least $n - \eta - w + 1$ $q$-digits $r_i$, $i \in [0, n-1]\setminus M$, each satisfying
$r_i \geq s_q(r)/w$ (otherwise the number of indices $i \in [0, n-1]\setminus M$ satisfying $r_i < s_q(r)/w$ is at least $w$. This gives a subset $W \subseteq [0, n-1]\setminus M$ of size $w$ 
for which there exists no $i \in W$ with $r_i \geq s_q(r)/w$, a contradiction). 
Since $s_q(r) = \sum_{i \in [0, n-1] \setminus M} r_i + (q-1)\eta$, we obtain
\begin{equation}\label{eqn: claim 2}
 s_q(r) \geq \dfrac{n - \eta - w + 1}{w}s_q(r) + (q-1)\eta.
\end{equation}
Necessarily $(n - \eta - w + 1)/w \leq 1$. Rearranging terms in (\ref{eqn: claim 2}) we get
\begin{equation}\label{eqn: claim 3}
(q-1) \eta \leq s_q(r) \left( 1 - \dfrac{n - \eta - w + 1}{w}\right).
\end{equation}
Now the fact that $s_q(r) \leq (q-2) w$ yields
\begin{equation}\label{eqn: claim 4}
 (q-1) \eta \leq (q-2) w \left( 1 - \dfrac{n - \eta - w + 1}{w}\right),
\end{equation}
which is equivalent to 
\begin{equation}\label{eqn: claim 5}
 \eta \leq (q-2)(2w - n - 1).
\end{equation}
Since $w \leq n/2$, this means that $\eta \leq -(q-2)$. Because $\eta \geq 0$, this however implies that $q = 2$ and $\eta = 0$. But then all digits of $r$ in its binary representation
are zero and hence $r = 0$, a contradiction.
The claim follows. 

Let $W \in \mathcal{C}$ such that
\begin{equation}\label{eqn: claim 6}
 s_q(r) > w \max_{i \in W} \{r_i\}.
\end{equation}
Thus
\begin{equation}\label{eqn: claim 7}
 s_q(r) + (q-1 - \max_{i \in W}\{r_i\})w > (q-1)w.
\end{equation}
Let $s = q-1 - \max_{i \in W}\{r_i\}$ and $t = \sum_{i \in W} q^i \in \Omega(w)$. Clearly $1 \leq s \leq q-1$ (since $r_i \leq q-2$ for all $i \in W \in \mathcal{C}$). 
Then $st \in \supp(\Delta_{w,c})$ and so $st + r \in \supp(\Delta_{w,c})$. Note also that $1 \leq s + r_i \leq q-1$ for each $i \in W$. 
Then there occurs no carry in the $q$-adic addition of $st$ and $ r$.
Hence
\begin{align}
 s_q(st + r) &= s_q(r) + s_q(st) = s_q(r) + (q-1 - \max_{i \in W}\{r_i\})w \nonumber\\
 &> (q-1)w.\label{eqn: claim 8}
\end{align}
Since $st, r < q^n-1$ as well, it follows (from the absence of carry) that $st + r \leq q^n-1$. 
If $st + r = q^n-1 \equiv 0 \pmod{q^n-1}$, then $\Delta_{w,c}(0) = \Delta_{w,c}(st + r)$ and so $0 \in \supp(\Delta_{w,c})$.
This contradicts the fact that $\delta_{w}^{\otimes k}(0) = 0$ for all $1 \leq k \leq q-1$ (since $1 \leq w < n$). 
Then $0 < st + r < q^n-1$ and $st + r = (st + r) \bmod(q^n-1)$. This in conjunction with (\ref{eqn: claim 8}) yields $s_q((st + r) \bmod(q^n-1) ) > (q-1)w$. 
But then $st + r \not\in \supp(\Delta_{w,c})$,
a contradiction. Thus no integer $r$ with $1 \leq r < q^n-1$ and $s_q(r) \leq (q-1)n/2$ can be a period of $\Delta_{w,c}$. 
Necessarily $s_q(r) > (q-1)n/2$. 
Let $r' = q^n-1 - r$. Clearly $\Delta_{w,c}$ is $r'$-periodic. However note that
$1 \leq r' < q^n-1$ and $s_q(r') < (q-1)n/2$, a contradiction. Necessarily $\Delta_{w,c}$ has maximum least period $q^n-1$. This concludes the proof for the case when $c \neq 0$.
\\
\\
{\bf Case 2 ($c = 0$):} Assume $c = 0$. 
Note $\Delta_{w,0} = \delta_0 - \delta_w^{\otimes(q-1)}$ and $\Delta_{w,0}(0) = 1$. 
Thus $\Delta_{w,0}(r) = \Delta_{w,0}(0 + r) = 1$. Since $0 < r < q^n-1$, necessarily $\delta_w^{\otimes(q-1)}(r) = -1$. 
In particular $r \in \supp(\delta_w^{\otimes(q-1)})$ and $s_q(r) = (q-1)w$. Because $1 \leq r < q^n-1$ is a period of $\Delta_{w,0}$, so is $r' = q^n-1 - r$ with $1 \leq r' < q^n-1$. 
Then the previous arguments similarly imply that $s_q(r') = (q-1)w$. Given that
$s_q(r') = (q-1)n - s_q(r)$, it follows $w = n/2$ and $s_q(r) = (q-1)n/2$. In particular $n$ is even and $\Delta_{w,0} = \Delta_{n/2, 0} = \delta_0 - \delta_{n/2}^{\otimes(q-1)}$. 

Consider the case when $n > 2$:
Suppose not all digits of $r$ are the same (since $s_q(r) = (q-1)n/2$, the last is equivalent to supposing that $r \neq (q^n-1)/2$; this is the case in particular when $q$ is even). 
Clearly either there exists $k \in [0, n-2]$ such that $r_k > r_{k+1}$ or the sequence $r_0, \ldots, r_{n-1}$ is non-decreasing. Suppose the former holds. 
Fix any such $k$ and let $\sigma$ be the permutation of $[0, n-1]$ which fixes each index in $[0, n-1]\setminus \{k, k+1\}$ and maps $k \mapsto k+1$ and $k+1 \mapsto k$. 
Thus 
\begin{equation}\label{eqn: sigma}
\varphi_\sigma(r) = r_k q^{k+1} + r_{k+1} q^k + \sum_{i \in [0, n-1]\setminus \{k, k+1\}} r_i q^i >  r_{k+1}q^{k+1} + r_kq^k  + \sum_{i \in [0, n-1]\setminus\{k, k+1\}} r_i q^i = r,
\end{equation}
since $r_k > r_{k+1}$. Because $\varphi_\sigma(r)$ is obtained via a permutation of the digits of $r$, and $0 < r < q^n-1$, then $0 < \varphi_\sigma(r) < q^n-1$. Now note
\begin{align}\label{eqn: simga 1}
 \varphi_\sigma(r) - r &= (r_k - r_{k+1})q^{k+1} - (r_k - r_{k+1})q^k \nonumber \\
 &= (r_k - r_{k+1} - 1)q^{k+1} + (q - (r_k - r_{k+1}))q^{k}.
\end{align}
Since $1 \leq r_k - r_{k+1} \leq q-1$, it follows the above coefficients are contained in the set $[0, q-1]$; hence this is the $q$-adic form of $\varphi_\sigma(r) - r$ and one can see that 
$s_q(\varphi_\sigma(r) - r) = q-1$.

Because $\delta_{n/2}$ is $q$-symmetric with $\epsilon_i(m) \leq 1$ for each $m \in \supp(\delta_{n/2}) = \Omega(n/2)$ and each $0 \leq i \leq n-1$, 
it follows from Lemma \ref{lem: convolution of q-sym is q-sym} that $\delta_{n/2}^{\otimes(q-1)}$ is $q$-symmetric. 
In particular $\delta_{n/2}^{\otimes(q-1)}(\varphi_\sigma(r)) = \delta_{n/2}^{\otimes(q-1)}(r)$. Since $\varphi_\sigma(r) \neq 0$, then 
$\Delta_{n/2, 0}(\varphi_\sigma(r)) = -\delta_{n/2}^{\otimes(q-1)}(\varphi_\sigma(r)) = -\delta_{n/2}^{\otimes(q-1)}(r) = 1$; hence $\varphi_\sigma(r) \in \supp(\Delta_{n/2, 0})$.
Given that $\Delta_{n/2, 0}$ is $r$-periodic, then $\varphi_\sigma(r) - r \in \supp(\Delta_{n/2, 0})$. Since $0 < \varphi_\sigma(r) - r < q^n-1$, 
then $\varphi_\sigma(r) - r \in \supp(\delta_{n/2}^{\otimes(q-1)})$. It follows $s_q(\varphi_\sigma(r) - r) = (q-1)n/2$, contradicting $s_q(\varphi_\sigma(r) - r) = q-1$ with $n > 2$.
Necessarily the $q$-digits $r_0, \ldots, r_{n-1}$ of $r$ must form a non-decreasing sequence. Since not all digits of $r$ are the same, in particular $r_{n-1} > r_0$.

Since $\Delta_{n/2, 0}$ is $r$-periodic, it is $r'' := (qr \bmod(q^n-1))$-periodic. Note $0 < r'' < q^n-1$ and 
$r'' = (r_{n-1}, r_0, r_1, \ldots, r_{n-2})_q$. However observe that $r_{0} = \epsilon_1(r'') < \epsilon_0(r'') = r_{n-1}$. 
Then we can reproduce the previous arguments with $r$ and $k$ substituted with $r''$ and $0$, respectively, to obtain a contradiction.
Thus for $n > 2$, it is impossible that $\Delta_{n/2, 0}$ is $r$-periodic if $0 < r < q^n-1$ and not all digits of $r$ are the same. 
In particular when $q$ is even and $n > 2$, 
$\Delta_{n/2, 0}$ must have maximum least period $q^n-1$. 

Note that at this point the proof of (i) is complete. In the case of (ii), with $q$ odd and $n > 2$,
we have shown that either $r = (q^n-1)/2$ (all digits of $r$ are the same) or no such $r$ with $0 < r < q^n-1$
can be a period of $\Delta_{n/2, 0}$ (when not all digits of $r$ are the same), whence the least period of $\Delta_{n/2, 0}$ must be the maximum, $q^n-1$. 
Thus the proof of (ii) is complete as well. 

Consider now the case, (iii), with $n = 2$ and $q$ odd: Here $w = n/2 = 1$ and we need to show $r > q-1$. On the contrary, suppose $r \leq q-1$. 
Since $s_q(r) = (q-1)n/2 = q-1$, it follows that $r = q-1$. Note there is exactly one way to write $r = q-1$ as a sum of $q-1$ ordered elements in $\Omega(1) = \{1, q\}$, 
namely as $q - 1 = 1 + \cdots + 1$, a total of 
$q-1$ times. Thus $\delta_{1}^{\otimes(q-1)}(r) = 1$. This contradicts the fact (see the beginning of the proof of Case 2) that $\delta_{1}^{\otimes(q-1)}(r) = -1$ with $q$ odd. 
This completes the proof of (iii) and of Case 2 here. 

It remains to notice from (i), (ii), (iii), that the least period $r$ of $\Delta_{w,c}$ satisfies $r > (q^n-1)/\Phi_n(q)$ in every case. 
Indeed, both (i), (ii) follow immediately from the fact that $\Phi_n(q) > q-1$ for $n \geq 2$. 
In the case of (iii), we have $r > q-1 = (q^2-1)/(q + 1) = (q^2-1)/\Phi_2(q)$ as well.
This concludes the proof of Lemma \ref{lem: period of delta and hansen-mullen}.
\end{proof}

\begin{proof}[{\bf Proof of Theorem \ref{thm: hansen-mullen}}]
It is elementary to show that every element of $\F^*$ is the norm of an element of degree $n$ over $\F$; see for example \cite{hansen-mullen}. Thus we may assume $w < n$. 
In view of the symmetry between the coefficients of a polynomial and 
 its reciprocal, as well as the fact that a polynomial is irreducible if and only if so is its reciprocal, we may further assume $1 \leq w \leq n/2$. Now the result follows from
 Lemma \ref{lem: period of delta and hansen-mullen} together with Lemma \ref{lem: delta function}. 
\end{proof}


\begin{thebibliography}{115}
\bibitem{Aigner}
M. Aigner and G.M. Ziegler, {\em Proofs from the book}, 4th ed.,
Springer-Verlag, Berlin, 2010.

\bibitem{Barvinok}
A. Barvinok, {\em Matrices with prescribed row and column sums}, Linear Algebra Appl.  436 (2012), no. 4,  820--844.

\bibitem{Bourgain}
J. Bourgain, {\em Prescribing the binary digits of primes, II}, Israel J. Math.
206 (2015), no. 1, 165--182. 
%doi:10.1007/s11856-014-1129-5. URL http://dx.doi.org/10.1007/s11856-014-1129-5


\bibitem{Canteaut}
A. Canteaut and M. Videau, {\em Symmetric boolean functions}, IEEE Transactions on Information
Theory, Institute of Electrical and Electronics Engineers 51 (2005), no. 8, 2791--2811.

\bibitem{Carlitz}
L. Carlitz, {\em A theorem of Dickson on irreducible polynomials}, Proc. Amer. Math. Soc. 3 (1952), 693--700.

\bibitem{Castro}
F.N. Castro and L.A. Medina, {\em Linear recurrences and asymptotic behavior of exponential sums of symmetric boolean functions}, Electr. J. Comb. 18 (2011), no. 2, paper \#P9.

\bibitem{Cohen 2004}
S.D. Cohen, {\em Primitive polynomials over small fields},  Finite fields and applications, 197--214, Lecture Notes in Comput. Sci., 2948, Springer, Berlin, 2004. 

% Finite Fields and Applications, in Lecture Notes in Computer Science 2948 (2004),  197--214.

\bibitem{cohen 2006}
S.D. Cohen, {\em Primitive polynomials with a prescribed coefficient}, Finite Fields Appl. 12 (2006), no. 3, 425--491.


\bibitem{cohen-presern 2006}
S.D. Cohen and M. Pre\v{s}ern, {\em Primitive polynomials with prescribed second
coefficient}, Glasgow Math. J. 48 (2006), 281--307.


\bibitem{cohen-presern 2008}
S.D. Cohen and M. Pre\v{s}ern, {\em The Hansen-Mullen primitivity conjecture:
completion of proof}, Number theory and polynomials, 89--120, London Math. Soc. Lecture Note Ser., 352, Cambridge Univ. Press, Cambridge, 2008.

%In Number theory and polynomials, volume 352 of LMS
%Lecture notes, pages 89--120. Cambridge University Press, Cambridge, 2008.

%\bibitem{Drmota-Rivat-Stoll}
%M. Drmota, J. Rivat, T. Stoll, {\em The sum of digits of primes in $\mathbb{Z}[i]$}, Monatshefte f\"ur Mathematik
%v.155 no.3, (2008), 317--347.

\bibitem{fan-han}
S.Q. Fan and W.B. Han,
{\em p-Adic formal series and primitive polynomials over finite fields}
Proc. Amer. Math. Soc. 132 (2004),  15--31.

\bibitem{Fitzgerald-Yucas}
R.W. Fitzgerald and J.L. Yucas, {\em Irreducible polynomials over GF(2) with three prescribed coefficients}, Finite Fields Appl. 9 (2003), 286--299.



\bibitem{garefalakis}
T. Garefalakis, {\em Irreducible polynomials with consecutive zero coefficients}, Finite Fields Appl. 14 (2008), no. 1, 201--208.

\bibitem{Ha}
J. Ha, {\em Irreducible polynomials with several prescribed coefficients}, arXiv:1601.06867 [math.NT], preprint (2016).

\bibitem{ham-mullen}
K.H. Ham and G.L. Mullen, {\em Distribution of irreducible polynomials of small degrees over finite fields}, Math. Comp. 67 (1998), no. 221, 337--341.

\bibitem{han}
W.B. Han, {\em On Cohen's problem}, Chinacrypt ’96, Academic Press (China) (1996) 231--235 (Chinese).

\bibitem{hansen-mullen}
T. Hansen and G.L. Mullen, {\em Primitive polynomials over finite fields}, Math. Comp. 59 (1992), 639--643.

\bibitem{Koc}
\c{C}.K. Ko\c{c}, {\em Open Problems in Mathematics and Computational Science}, Springer International Publishing, 2014.

\bibitem{KMRV}
K. Kononen, M. Moisio, M. Rinta-aho, K. V\"{a}\"{a}n\"{a}nen, {\em Irreducible polynomials with prescribed trace and restricted norm},
JP J. Algebra Number Theory Appl. 11 (2009), 223--248.


\bibitem{Lint}
J.H. van Lint and R.M. Wilson, {\em A course in combinatorics}, 2nd ed., Cambridge
University Press, Cambridge, 2001.

\bibitem{KPW}
B. Omidi Koma, D. Panario, Q. Wang, {\em The number of irreducible polynomials of degree $n$ over $\F$ with given trace and
constant terms}, Discrete Math. 310 (2010), 1282--1292.


\bibitem{George paper}
D. Panario and G. Tzanakis,
{\em A generalization of the Hansen--Mullen conjecture on irreducible polynomials over finite fields}, Finite Fields Appl. 18 (2) (2012) 303--315.

%\bibitem{Mauduit-Rivat}
%C. Mauduit, J. Rivat, {\em Sur un probl\`eme de Gelfond : la somme des chiffres des nombres premiers}, Annals of Mathematics, Vol. 171 (2010), No. 3, 1591--1646.

%\bibitem{Morgenbesser}
%J.F. Morgenbesser, {\em The sum of digits of Gaussian primes}, The Ramanujan Journal, v.27 no.1 (2012) 43--70.

\bibitem{TW}
A. Tuxanidy and Q. Wang, {\em On the number of $N$-free elements with prescribed trace}, J. Number Theory 160 (2016), 536--565.

\bibitem{pollack}
P. Pollack, {\em Irreducible polynomials with several prescribed coefficients}, Finite Fields Appl. 22 (2013), 70--78.

\bibitem{ren}
D-B. Ren, {\em On the coefficients of primitive polynomials over finite fields}, Sichuan
Daxue Xuebao 38 (2001), 33--36.

\bibitem{Shparlinski}
I. E. Shparlinski, {\em On primitive polynomials}, Prob. Peredachi Inform. 23 (1988),
100--103 (Russian).

%\bibitem{tuxanidy-wang}
%A. Tuxanidy, Q. Wang, {\em Composed products and factors of cyclotomic polynomials over finite fields}, Des. Codes Cryptogr. 69 (2013), 203--231.

%\bibitem{Tuxanidy-Wang, H-M conjecture}
%A. Tuxanidy, Q. Wang. {\em A new proof of the Hansen-Mullen irreducibility conjecture}, preprint (2016).

\bibitem{George thesis}
G. Tzanakis, {\em On the existence of irreducible polynomials with prescribed coefficients over finite fields}, Master's thesis,
Carleton University, 2010, http://www.math.carleton.ca/~gtzanaki/mscthesis.pdf.

\bibitem{wan}
D. Wan, {\em Generators and irreducible polynomials over finite fields}, Math. Comp. 66 (1997), no. 219, 1195--1212.






\end{thebibliography}
\end{document}